\documentclass[12pt]{amsart}
\usepackage{mathptmx}
\usepackage{amssymb}
\usepackage{amsfonts}
\usepackage{amsmath}
\usepackage{graphicx}
\usepackage[pagebackref]{hyperref}

\newtheorem{thm}{Theorem}[section]
\newtheorem{corollary}[thm]{Corollary}

\newtheorem{proposition}[thm]{Proposition}

\theoremstyle{definition}

\newtheorem{example}[thm]{Example}

\theoremstyle{remark}

\DeclareMathOperator{\htt}{ht}
\DeclareMathOperator{\qf}{qf}
\DeclareMathOperator{\td}{t.d.}

\newcommand{\field}[1]{\mathbb{#1}}

\newcommand{\N}{\field{N}}
\def\1{{\rm (1)}}
\def\2{{\rm (2)}}
\def\3{{\rm (3)}}
\def\4{{\rm (4)}}
\def\5{{\rm (5)}}

\makeatletter
  \newcounter{xenumi}
  
\makeatother
\usepackage{remreset}
\makeatletter\@removefromreset{equation}{section}\makeatother

\begin{document}

\title[Bouvier's conjecture]{Bouvier's conjecture}

\author{S. Bouchiba}

\address{Department of Mathematics,
Faculty of Sciences, University of Meknes, Meknes 50000, Morocco}

\email{bouchiba@fsmek.ac.ma}

\author{S. Kabbaj}

\address{Department of Mathematical Sciences,
King Fahd University of Petroleum \& Minerals, P. O. Box 5046,
Dhahran 31261, Saudi Arabia}

\email{kabbaj@kfupm.edu.sa}

\thanks{This work was funded by King Fahd University of Petroleum \& Minerals under \# IP-2007/39.}

\date{}



\begin{abstract}
This paper deals with Bouvier's conjecture which sustains that finite-dimensional non-Noetherian Krull domains need not be Jaffard.
\end{abstract}

\maketitle

\section{Introduction}\label{Int}

All rings and algebras considered in this paper are commutative with identity element and, unless otherwise specified, are assumed to be non-zero. All ring homomorphisms are unital. If $k$ is a field and $A$ a domain which is a $k$-algebra, we use $\qf(A)$ to denote the quotient field of $A$ and $\td(A)$ to denote the transcendence degree of $\qf(A)$ over $k$. Finally, recall that an affine domain over a ring $A$ is a finitely generated $A$-algebra that is a domain \cite[p. 127]{Na2}. Any unreferenced material is standard as in \cite{G1,Ka,Ma}.

A finite-dimensional integral domain $R$ is said to be Jaffard if $$\dim(R[X_{1}, ..., X_{n}])= n + \dim(R)$$ for all $n \geq 1$; equivalently, if $\dim(R) = \dim_{v}(R)$, where $\dim(R)$ denotes the (Krull) dimension of $R$ and $\dim_{v}(R)$ its valuative dimension (i.e., the supremum of dimensions of the valuation overrings of $R$). As this notion does not carry over to localizations, $R$ is said to be locally Jaffard if $R_p$ is a
Jaffard domain for each prime ideal $p$ of $R$ (equiv., $S^{-1}R$ is a Jaffard domain for each multiplicative subset $S$ of $R$). The class of Jaffard domains contains most of the well-known classes of rings involved in Krull dimension theory such as Noetherian domains, Pr\"ufer domains, universally catenarian domains, and universally strong S-domains. We assume familiarity with these concepts, as in \cite{ABDFK,ADKM,BDF,BK,DFK,J,K1,K2,MM}.

It is an open problem to compute the dimension of polynomial rings over Krull domains in general. In this vein, Bouvier conjectured that
``finite-dimensional Krull (or more particularly factorial) domains need not be Jaffard'' \cite{BK,FK}. In Figure~\ref{D}, a diagram of implications
places this conjecture in its proper perspective and hence shows how it naturally arises. In particular, it indicates how the classes of
(finite-dimensional) Noetherian domains, Pr\"ufer domains, UFDs, Krull domains, and PVMDs \cite{G1} interact with the notion of Jaffard
domain as well as with the (strong) S-domain properties of Kaplansky \cite{K2,Ka,MM}.

\begin{figure}
\[\setlength{\unitlength}{.9mm}
\begin{picture}(100,80)(0,0)
\put(40,80){\vector(0,-2){20}} \put(40,80){\vector(-2,-1){40}}
\put(40,80){\vector(2,-1){40}} \put(0,60){\vector(0,-2){20}}
\put(40,60){\vector(-2,-1){40}} \put(40,60){\vector(0,-2){20}}
\put(40,60){\vector(2,-1){40}} \put(80,60){\vector(0,-2){20}}
\put(80,20){\vector(0,-2){20}} \put(40,40){\vector(-2,-1){40}}
\put(40,40){\vector(2,-1){40}} \put(40,40){\vector(0,-2){20}}
\put(0,20){\vector(0,-2){20}} \put(80,40){\vector(0,-2){20}}
\put(80,40){\vector(-2,-1){40}} \put(0,40){\vector(0,-2){20}}
\put(40,20){\vector(-2,-1){40}} \put(40,20){\vector(2,-1){40}}

\put(-8,40){\oval(18,7)} \put(-8,60){\oval(18,7)}
\put(62,26){\oval(36,7)} \put(99,20){\oval(35,7)}
\put(90,10){\oval(18,7)}

\put(40,80){\circle*{.7}} \put(40,82){\makebox(0,0)[b]{PID}}
\put(0,60){\circle*{.7}} \put(-2,60){\makebox(0,0)[r]{UFD}}
\put(0,40){\circle*{.7}} \put(-2,40){\makebox(0,0)[r]{Krull}}
\put(0,20){\circle*{.7}} \put(-2,20){\makebox(0,0)[r]{PVMD}}
\put(40,60){\circle*{.7}} \put(43,60){\makebox(0,0)[l]{Dedekind}}
\put(40,40){\circle*{.7}} \put(43,40){\makebox(0,0)[l]{Pr\"ufer}}
\put(40,20){\circle*{.7}} \put(40,16){\makebox(0,0)[t]{Strong
S-domain}}
\put(80,20){\circle*{.7}}\put(84,20){\makebox(0,0)[l]{Locally
Jaffard}}
\put(80,10){\circle*{.7}}\put(84,10){\makebox(0,0)[l]{Jaffard}}
\put(60,30){\circle*{.7}}\put(45,26){\makebox(0,0)[l]{\small U.
Strong S-Domain}}
\put(80,0){\circle*{.7}}\put(80,-2){\makebox(0,0)[t]{$\dim R[X]
=\dim R + 1$}}
\put(0,0){\circle*{.7}}\put(0,-2){\makebox(0,0)[t]{S-domain}}
\put(80,40){\circle*{.7}} \put(82,40){\makebox(0,0)[l]{Noetherian}}
\put(80,60){\circle*{.7}} \put(82,60){\makebox(0,0)[l]{UFD +
Noetherian}}
\end{picture}\]\medskip

\caption{\tt Diagram of Implications} \label{D}
\end{figure}
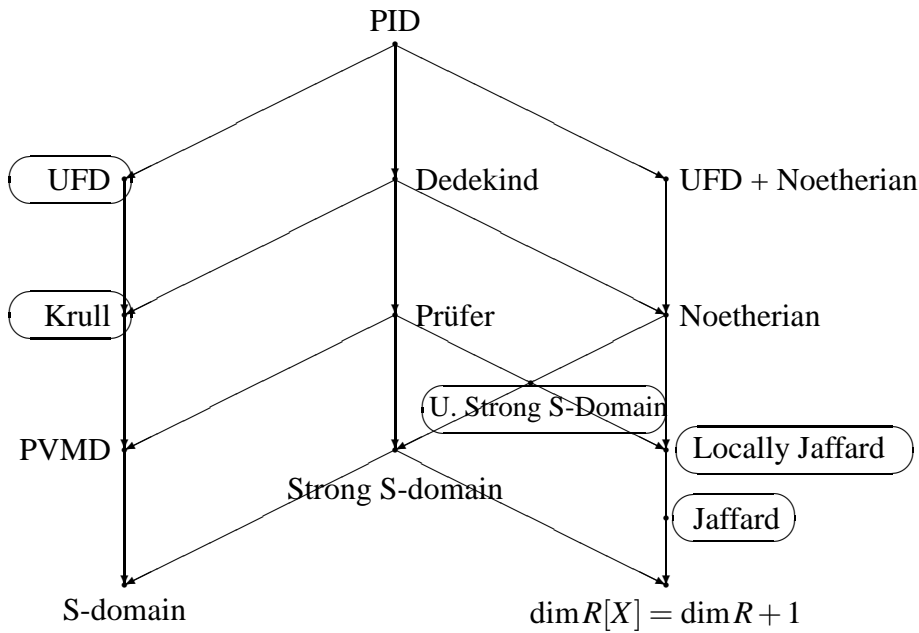

This paper scans all known families of examples of non-Noetherian finite dimensional  Krull (or factorial) domains existing in the literature. In Section 2, we show that most of these examples are in fact locally Jaffard domains. One of these families which arises from David's second example \cite{Da2} yields examples of Jaffard domains but it is still open whether these are locally Jaffard. Further, David's example turns out to be the first example of a $3$-dimensional factorial domain which is not catenarian (i.e., prior to Fujita's example \cite{Fu}). Section 3 is devoted to the last known family of examples which stem from the generalized fourteenth problem of Hilbert (also called Zariski-Hilbert problem): Let $k$ be a field of characteristic zero, $T$ a normal affine domain over $k$, and $F$ a subfield of $\qf(T)$. The Hilbert-Zariski problem asks whether $R:=F\cap T$ is an affine domain over $k$. Counterexamples on this problem were constructed by Rees \cite{Re}, Nagata \cite{Na1} and Roberts \cite{R1,R2} where $R$ wasn't even Noetherian. In this vein, Anderson, Dobbs, Eakin, and Heinzer \cite{ADEH} asked whether $R$ and its localizations inherit from $T$ the Noetherian-like main behavior of having Krull and valuative dimensions coincide (i.e., Jaffard). This problem will be addressed within the more general context of subalgebras of affine domains over Noetherian domains; namely, let $A\subseteq R$ be an extension of domains where $A$ is Noetherian and $R$ is a subalgebra of an affine domain over $A$. It turns out that $R$ is Jaffard but it is still elusively open whether $R$ is locally Jaffard.

\section{Examples of non-Noetherian Krull domains}\label{sec:2}

Obviously, Bouvier's conjecture (mentioned above) makes sense beyond the Noetherian context. As the notion of Krull domain is stable under formation of
rings of fractions and adjunction of indeterminates, it merely claims ``the existence of a Krull domain $R$ and a multiplicative subset $S$ (possibly
equal to $\{1\}$) such that \(1+{\tt dim}(S^{-1}R)\lneqq {\tt dim}(S^{-1}R[X])."\) However, finite-dimensional non-Noetherian Krull domains are
scarce in the literature and one needs to test them and their localizations as well for the Jaffard property.

Next, we show that most of these families of examples are subject to the (locally) Jaffard property. This reflects the difficulty of proving or disproving Bouvier's conjecture.

\begin{example}\label{sec:2.1}
 Nagarajan's example \cite{Nj} arises as the ring $R_0$ of invariants of a finite group
 of automorphisms acting on $R:=k[[X,Y]]$, where $k$ is a field of
 characteristic $p\not=0$. It turned out that $R$ is integral
 over $R_0$. Therefore \cite[Theorem 4.6]{MM} forces $R_{0}$ to be a universally strong
 S-domain, hence a locally Jaffard domain \cite{ABDFK,Ka}.
\end{example}

\begin{example}\label{sec:2.2}
Nagata's example \cite[p. 206]{Na2} and David's example \cite{Da1}
arise as integral closures of Noetherian domains, which are
necessarily universally strong S-domains by \cite[Corollary
4.21]{MM}
 (hence locally Jaffard).
\end{example}

\begin{example}\label{sec:2.3}
Gilmer's example \cite{G2} and Brewer-Costa-Lady's example \cite{BCL} arise as group rings (over a field and a group of finite rank), which are
universally strong S-domains by \cite{ACKZ} (hence locally Jaffard).
\end{example}

\begin{example}\label{sec:2.4}
Fujita's example \cite{Fu} is a
3-dimensional factorial quasilocal domain $(R,M)$ that arises as a
directed union of 3-dimensional Noetherian domains, say $R=\bigcup
R_{n}$. We claim $R$ to be a locally Jaffard domain.

Indeed, the localization with respect to any height-one prime ideal is a DVR
(i.e., discrete valuation ring) and hence a Jaffard domain. As, by
\cite[Theorem 2.3]{DFK}, $R$ is a Jaffard domain, then $R_M$ is
locally Jaffard. Now, let $P$ be a prime ideal of $R$ with
$\htt(P)=2$. Clearly, there exists $Q\in$ Spec$(R)$ such that
$(0)\subset Q\subset P\subset M$ is a saturated chain of prime
ideals of $R$. As, $\htt(M[n])=\htt(M)=3$ for each positive integer $n$,
we obtain $\htt(P[n])=\htt(P)=2$ for each positive integer $n$. Then
$R_P$ is locally Jaffard, as claimed.
\end{example}

\begin{example}\label{sec:2.5}
David's second example \cite{Da2} is a $3$-dimensional factorial domain
$J:=\bigcup J_n$ which arises as an
ascending union of $3$-dimensional polynomial rings $J_n$ in three
indeterminates over a field $k$. We claim that $J$ is a Jaffard
domain. Moreover, $J$ turns out to be non catenarian. Thus, David's example is the first example of a
$3$-dimensional factorial domain which is not catenarian (prior to Fujita's example).

Indeed, we have $J_n:=k[X,\beta_{n-1},\beta_n]$ for each positive integer $n$, where the indeterminates $\beta_n$ satisfy the following condition: For $n\geq 2$,
\begin{equation}\label{eq2.5.1}\beta_n=\displaystyle {\frac {-\beta_{n-1}^{s(n)}+\beta_{n-2}}X}\end{equation}
where the $s(n)$ are positive integers. Also, $J_n\subseteq
J\subseteq J_n[X^{-1}]$ for each positive integer $n$. By \cite[Theorem 2.3]{DFK}, $J$ is a Jaffard domain, as the $J_n$ are affine domains. Notice, at this point, we weren't able to prove or disprove that $J$ is locally Jaffard.

Next, fix a positive integer $n$. We have ${\dfrac {J_n}{XJ\cap
J_n}}=k[\overline {\beta_{n-1}},\overline {\beta_n}]$. On account of (\ref{eq2.5.1}), we get
\begin{equation}\label{eq2.5.2}\overline {\beta_{n-1}}=\overline{\beta_n}^{\ s(n+1)}.\end{equation}
Therefore
\begin{equation*}\displaystyle {\frac {J_n}{XJ\cap J_n}}=k[\overline{\beta_n}].\end{equation*}
Iterating the formula in (\ref{eq2.5.2}), it is clear that for
each positive integers $n\leq m$, there exists a positive integer $r$
such that $\overline {\beta_n}=\overline {\beta_m}^{\ r}$ with respect
to the integral domain $\displaystyle {\frac J{XJ}}$. It follows
that $\displaystyle {\frac J{XJ}}$ is integral over $k[\overline
{\beta_n}]$ for each positive integer $n$. Surely, $\overline
{\beta_n}$ is transcendental over $k$, for each positive integer $n$,
since $(0)\subset XJ\subset M:=(X,\beta_0,\beta_1,...,\beta_n,...)$
is a chain of distinct prime ideals of $J$. Then dim$(\displaystyle
{\frac J{XJ}})=1$ and thus $(0)\subset XJ\subset
M:=(X,\beta_0,\beta_1,...,\beta_n,...)$ is a saturated chain of
prime ideals of $J$. As $\htt(M)=3$, it follows that $J$ is not
catenarian, as desired.
\end{example}

\begin{example}\label{sec:2.6}
Anderson-Mulay's example \cite{AM}
draws from a combination of
 techniques of Abhyankar \cite{Ab} and Nagata \cite{Na2} and arises as a
 directed union of polynomial rings
 over a field. Let $k$ be
 a field, $d$ an integer $\geq1$, and $X,Z,Y_{1}, ..., Y_{d}$  $d+2$
 indeterminates over $k$. Let $\{\beta_{i}:=\sum_{n\geq 0} b_{in}X^{n}\ |\ 1\leq i\leq d\}\subset
 k[[X]]$ be a set of algebraically independent elements over $k(X)$ (with $b_{in}\not=0$ for all $i$ and $n$).
 Define $\{U_{in}\ |\ 1\leq i\leq d,\ 0\leq
 n\}$ by
 \begin{eqnarray*}
  U_{i0}&:= &Y_{i}\\
  U_{in}&:= &\frac{Y_{i}+Z(\sum_{0\leq k\leq n-1}
 b_{ik}X^{k})}{X^{n}},\ \mbox{for}\ n\geq 1.
\end{eqnarray*}
 For any $i, n$ we have
\begin{equation}\label{eq2.6.1}
  U_{in}=XU_{i(n+1)}-b_{in}Z.
\end{equation}
Let $R_{n}:= k[X,Z,U_{1n}, ..., U_{dn}]$, a polynomial ring in $d+2$
indeterminates (by~(\ref{eq2.6.1})); and
 let $R:=\bigcup R_{n}=k[X,Z,\{U_{1n}, ..., U_{dn}\ |\ n\geq0\}]$. They
 proved that $R$ is a $(d+2)$-dimensional non-Noetherian Jaffard and factorial domain.
We claim that $R$ is locally Jaffard. For this purpose, we envisage two cases.

{\bf Case 1: {\boldmath $k$} is
algebraically closed}. Let $P$ be a prime
  ideal of $R$. We may suppose $\htt(P)\geq2$ (since $R$ is factorial).
  Assume $X\notin P$. Clearly, $R_{0}\subset R\subset R_{0}[X^{-1}]$,
  then $R_{P}\cong (R[X^{-1}])_{PR[X^{-1}]}=(R_{0}[X^{-1}])_{PR_{0}[X^{-1}]}$
  is Noetherian (hence Jaffard). Assume $X\in P$. By~(\ref{eq2.6.1}), $\displaystyle {\frac{R}{XR}}\cong
  k[Z]$. Then $P=(X,f)$ for some irreducible polynomial $f$ in
  $k[Z]$. As $k$ is algebraically closed, we get $f=Z-\alpha$ for
  some $\alpha\in k$. For any positive integer $n$ and $i=1,...,d$, define
\begin{equation*}
V_{in}:=U_{in}+b_{in}\alpha.
\end{equation*}
Observe that, for each $n$ and $i$, we have
\begin{eqnarray*}
  R_{n}&=&k[X,Z-\alpha, V_{1n},...,V_{dn}]\\
  V_{in}&=&XU_{i(n+1)}-b_{in}(Z-\alpha).
 \end{eqnarray*}
Then $P\cap
  R_n=(X,Z-\alpha,\{V_{1n},...,V_{dn}\})$ is a maximal ideal of $R_n$ for
  each positive integer $n$. For each $0\leq i\leq d$, set
\begin{equation*}
  P_i:=(Z-\alpha,\{V_{rn}\}_{1\leq r\leq i,\ 0\leq n})R.
\end{equation*}
Each $P_i$ is
  a prime ideal of $R$ since $P_i\cap R_n=(Z-\alpha, V_{1n}, \ldots, V_{in} )$ is a prime ideal of $R_n$. This gives rise to the following
  chain of prime ideals of $R$
 \begin{equation*}0\subset (Z-\alpha)R=P_0\subset P_1\subset ...\subset P_d\subset P.\end{equation*}
  Each inclusion is proper since
  the $P_i$'s contract to distinct ideals in each $R_n$. Hence
  $\htt(P)\geq d+2$, whence $\htt(P)=d+2$ as $\dim(R)=d+2$. Since $R$
  is a Jaffard domain, we get $\htt(P[n])=\htt(P)$ for each positive
  integer $n$. Therefore, $R$ is locally Jaffard, as desired.

{\bf Case 2: {\boldmath $k$} is an arbitrary field}. Let $K$ be an
algebraic closure of $k$. Let $T_n=K[X,Z,U_{1n}, ..., U_{dn}]$ for
each positive integer $n$ and let $$T:=\bigcup_{n\geq 0}
T_{n}=K[X,Z,\{U_{1n}, ..., U_{dn}:n\geq0\}].$$ Let $Q$ be a minimal
prime ideal of $PT$. Then $Q=(X,Z-\beta)$ with $\beta\in K$, as
$\displaystyle {\frac T{XT}\cong K[Z]}$. By the above case, we have
$\htt(Q)=d+2$. Hence $\htt(PT)=d+2$. As $T_n\cong K\otimes_kR_n$, we get,
$$T=\displaystyle {\bigcup_{n\geq 0}T_n=\bigcup_{n\geq
0}K\otimes_kR_n=K\otimes_k\bigcup_{n\geq 0}R_n=K\otimes_kR}.$$
Then $T$ is a free and hence faithfully flat $R$-module. A well-known property of faithful flatness shows that $PT\cap R=P$. Further,
$T$ is an integral and flat extension of $R$. It follows that $\htt(PT)=\htt(P)=d+2$, and thus $R_P$ is
a Jaffard domain.
\end{example}

\begin{example}\label{sec:2.7}
Eakin-Heinzer's 3-dimensional non-Noetherian Krull domain, say $R$,
arises -via \cite{Re} and \cite[Theorem 2.2]{EH}- as the symbolic
Rees algebra with respect to a minimal prime ideal $P$ of the
2-dimensional homogeneous coordinate ring $A$ of a nonsingular
elliptic cubic defined over the complex numbers. We claim that this
construction, too, yields locally Jaffard domains. Indeed, let
$K:=\qf(A)$, $t$ be an indeterminate over $A$, and $P^{(n)}:=
P^{n}A_{P}\cap A$, the $n$th symbolic power of $P$, for $n\geq2$.
Set $R:=A[t^{-1},Pt,P^{(2)}t^{2},...,P^{(n)}t^{n},...]$, the
3-dimensional symbolic Rees algebra with respect to $P$. We have
$$A\subset A[t^{-1}]\subset R\subset A[t,t^{-1}]\subset K(t^{-1}).$$
Let $Q$ be a prime ideal of $R$, $Q':=Q\cap A[t^{-1}]$, and $q:=Q\cap A=Q'\cap A$. We envisage three cases.

{\bf Case 1:} {\boldmath$\htt(Q)=1$}. Then $R_{Q}$ is a DVR hence a Jaffard domain.

{\bf Case 2:} {\boldmath$\htt(Q)=3$}. Then $3=\dim(R_{Q})\leq \dim_{v}(R_{Q})\leq \dim_{v}(A[t^{-1}]_{Q'})=\dim(A[t^{-1}]_{Q'})\leq \dim(A[t^{-1}]=1+\dim(A)=3$. Hence $R_{Q}$
is a Jaffard domain.

{\bf Case 3:}  {\boldmath$\htt(Q)=2$}. If $t^{-1}\notin Q$, then $R_{Q}$ is a localization of $A[t,t^{-1}]$, hence a Jaffard domain. Next, assume that
$t^{-1}\in Q$. If $Q$ is a homogeneous prime ideal, then
$Q\subset M:=(m[t^{-1}]+t^{-1}A[t^{-1}])\oplus pt\oplus...\oplus
p^{(n)}t^n\oplus ...$ and $\htt(M)=3$, where $m$ is the unique maximal
ideal of $A$. As $R$ is a Jaffard domain, we get
$\htt(M[X_1,...,X_n])=\htt(M)=3$ for each positive integer $n$. Hence
$\htt(Q[X_1,...,X_n])=\htt(Q)=2$ for each positive integer $n$, so that
$R_Q$ is Jaffard. Now, assume that $Q$ is not homogeneous. As
$t^{-1}\in Q$ and $\htt(Q)=1+\htt(Q^*)$, where $Q^*$ is the ideal
generated by all homogeneous elements of $Q$, we get $Q^*=t^{-1}R$
which is a height one prime ideal of the Krull domain $R$. Also, for
each positive integer $n$, note that
$Q[X_1,X_2,...,X_n]^*=Q^*[X_1,...,X_n]$. Therefore, for each positive integer $n$, we have
\begin{eqnarray*}
\htt(Q[X_1,...,X_n])&=&1+\htt(Q[X_1,...,X_n]^*)\\
&=&1+\htt(Q^*[X_1,...,X_n])\\
&=&1+\htt(t^{-1}R[X_1,...X_n])\\
&=&1+\htt(t^{-1}R)=2\\
&=&\htt(Q).
\end{eqnarray*}
It follows that $R_Q$ is Jaffard, completing the proof. Notice that
Anderson-Dobbs-Eakin-Heinzer's example \cite[Example 5.1]{ADEH} is a
localization of $R$ (by a height 3 maximal ideal), then locally
Jaffard.
\end{example}

Also, Eakin-Heinzer's second example \cite{EH} is a universally strong S-domain; in fact, it belongs to the same family as Example~\ref{sec:2.1}.
Another family of non-Noetherian finite-dimensional Krull domains stems from the generalized fourteenth problem of Hilbert (also called
Zariski-Hilbert problem). This is the object of our investigation in the following section.

\section{Krull domains issued from the Hilbert-Zariski problem}\label{sec:3}

Let $k$ be a field of characteristic zero and let $T$ be a normal affine domain over $k$. Let $F$ be a subfield of the field of fractions of $T$. Set
$R:=F\cap T$. The Hilbert-Zariski problem asks whether $R$ is an affine domain over $k$. Counterexamples on this problem were constructed by Rees
\cite{Re}, Nagata \cite{Na1} and Roberts \cite{R1,R2}, where it is shown that $R$ does not inherit the Noetherian property from $T$ in general. In this vein,
Anderson, Dobbs, Eakin, and Heinzer \cite{ADEH} asked whether $R$ inherits from $T$ the Noetherian-like main behavior of being locally Jaffard. We
investigate this problem within a more general context; namely, extensions of domains $A\subseteq R$, where $A$ is Noetherian and $R$ is a subalgebra
of an affine domain over $A$.

The next result characterizes the subalgebras of affine domains over a Noetherian domain. It allows one to reduce the study of the prime ideal structure of these constructions to those domains $R$ between a Noetherian domain $B$ and its localization $B[b^{-1}]$ $(0\not=b\in B)$.

\begin{proposition}\label{sec:3.1}
Let $A\subseteq R$ be an extension of domains where $A$ is Noetherian. Then the following statements are equivalent:
\begin{enumerate}
\item  $R$ is a subalgebra of an affine domain over $A$;
\item  There is $r\neq 0\in R$ such that $R[r^{-1}]$ is an affine domain over $A$;
\item  There is an affine domain $B$ over $A$ and $b\neq 0\in B$ such that $B\subseteq R\subseteq B[b^{-1}]$.
\end{enumerate}
\end{proposition}

\begin{proof}
\1 $\Rightarrow$ \2 This is \cite[Proposition 2.1(b)]{Gi}.

\2 $\Rightarrow$ \3 Let $r\neq 0\in R$ and $x_1,...,x_n\in R[r^{-1}]$ such that $R[r^{-1}]=A[x_1,...,x_n]$. For each $i=1,...,n$, write
$x_i=\sum_{j=0}^{n_i}r_{ij}r^{-j}$ with $r_{ij}\in R$ and $n_i\in \N$. Let $B:=A[\{r_{ij}: i=1,...,n$ and $j=0,...,n_i\}]$ and let $b:=r$. Clearly, $B$ is an affine domain over $A$ such that $B\subseteq R\subseteq B[b^{-1}]$.

The implication \3 $\Rightarrow$ \1 is trivial, completing the proof of the proposition.
\end{proof}

\begin{corollary}\label{sec:3.2}
Let $A\subseteq R$ be an extension of domains where $A$ is Noetherian and $R$ is a subalgebra of an affine domain over $A$. Then there exists an
affine domain $T$ over $A$ such that $R\subseteq T$ and $R_p$ is Noetherian (hence Jaffard) for each prime ideal $p$ of $R$ that survives in $T$.
\end{corollary}

\begin{proof} By Proposition~\ref{sec:3.1}, there exists an affine domain $B$ over $A$ and a nonzero element $b$ of $B$ such that
$B\subseteq R\subseteq B[b^{-1}]$. Put $T=B[b^{-1}]$. Let $p$ be a prime ideal of $R$ that survives in $T$ (i.e., $b\not\in p$). Then it is easy to
see that $$R_p\cong R[b^{-1}]_{pR[b^{-1}]}=B[b^{-1}]_{pB[b^{-1}]}=T_{pT}$$ is a Noetherian domain, as desired.
\end{proof}

\begin{corollary}\label{sec:3.3}
Let $R$ be a subalgebra of an affine domain $T$ over a field $k$. Then:
\begin{enumerate}
\item $\dim(R)= \td(R)$ and $R$ is a Jaffard domain.
\item $\dim(R)= \htt(P\cap R)+\td(\frac{R}{P\cap R})$ for each prime ideal $P$ of $T$. In particular, $\dim(R)= \htt(M)$ for each maximal ideal $M$ of $R$ that survives in $T$.
\end{enumerate}
\end{corollary}

\begin{proof}
\1 This is \cite[Proposition 5.1]{CA} which is a consequence of a more general result on valuative radicals \cite[Th\'eor\`eme 4.4]{CA}. Also the statement ``$\dim(R)= \td(R)$'' is \cite[Corollary 1.2]{OY}. We offer here an alternate proof: By Proposition~\ref{sec:3.1}, there exists an affine domain $B$ over $k$ and a nonzero element $b$ of $B$ such that $B\subseteq R\subseteq B[b^{-1}]$. By \cite[Corollary 14.6]{Na2},
$\dim(B[b^{-1}])=\dim_v(B[b^{-1}])=\dim_v(B)=\dim(B)=\td(B)=\td(R)$. Further, observe that $B[b^{-1}]=R[b^{-1}]$ is a localization of $R$. Hence
$\dim(B[b^{-1}])=\dim(R[b^{-1}])\leq \dim(R)\leq\dim_v(R)\leq\dim_v(B)$. Consequently, $\dim(R)= \dim_v(R)=\td(R)$, as desired.

\2 Let $P$ be a prime ideal of $T$ with $p:=P\cap R$. By \cite[Th\'eor\`eme 1.2]{CA}, the extension
$R\subseteq T$ satisfies the altitude inequality formula. Hence
$$ht(P)+\mbox {t.d.}(\frac TP:\frac R{p})\leq ht(p)+\mbox {t.d.}(T:R).$$
By \cite[Corollary 14.6]{Na2}, we obtain
$$\mbox {t.d.}(T:k)-\mbox {t.d.}(\frac Rp:k)\leq ht(p)+\mbox { t.d.}(T:k)-\mbox {t.d.}(R:k).$$ Then t.d.$(R)\leq
ht(p)+$t.d.$(\displaystyle {\frac Rp}:k)$. Moreover, it is well
known that $$ht(p)+\mbox {t.d.}(\frac Rp:k)\leq \mbox {t.d.}(R)
\mbox { \cite[p. 10]{ZS}.}$$ Applying (1), we get $$\mbox
{dim}(R)=\mbox {t.d.}(R:k)=ht(p)+\mbox {t.d.}(\frac R{p}:k).$$

Finally, notice that if $M\in$ Spec$(R)$ with $MT\neq T$, then there
exists $M^{\prime}\in$ Spec$(T)$ contracting to $M$, so that
$$\mbox {t.d.}(\frac RM)\leq\mbox { t.d.}(\frac T{M^{\prime}})=0\mbox { \cite[Corollary
14.6]{Na2}},$$ completing the proof.
\end{proof}

The above corollaries shed some light on the dimension and prime ideal structure of the non-Noetherian Krull domains emanating from the Hilbert-Zariski problem. In particular, these are necessarily Jaffard. But we are unable to prove or disprove if they are locally Jaffard. An in-depth study is to be carried out on (some contexts of) subalgebras of affine domains over Noetherian domains in line with Rees, Nagata, and Roberts constructions.


\end{document}